\documentclass[a4paper,12pt]{article}
\usepackage{amsmath,amssymb}
\usepackage{parskip,url}
\usepackage{cancel}
\usepackage{comment}
\newtheorem{theorem}{Theorem}[section]
\newtheorem{lemma}[theorem]{Lemma}
\newtheorem{proposition}[theorem]{Proposition}

\newcommand{\ch}{\operatorname{char}}

\newcommand{\gr}{\operatorname{gr}}

\newcommand{\Sym}{\operatorname{Sym}}
\newcommand{\FM}{\operatorname{FM}}
\newcommand{\FG}{\operatorname{FG}}

\newcommand{\GKdim}{\operatorname{GKdim}}
\newcommand{\clKdim}{\operatorname{clKdim}}

\newenvironment{proof}{\par\noindent{\bf Proof.}}{$\qed$\par\bigskip}
\newcommand{\qed}{\enspace\vrule  height6pt  width4pt  depth2pt}
\begin{document}

\title{Group algebras and semigroup algebras defined by permutation relations of fixed length}

\author{Ferran Ced\'o\footnote{Research partially supported by a grant of MICIIN (Spain)
MTM2011-28992-C02-01.} \and Eric Jespers\footnote{Research supported
in part by Onderzoeksraad of Vrije Universiteit Brussel and Fonds
voor Wetenschappelijk Onderzoek (Belgium).} \and Georg Klein}
\date{}
\maketitle

\begin{abstract}
Let $H$ be a subgroup of $\Sym_n$, the symmetric group of degree
$n$. For a fixed integer $l \geq 2$, the group $G$ presented
with generators $x_1, x_2, \ldots ,x_n$ and with relations
$x_{i_1}x_{i_2}\cdots x_{i_l} =x_{\sigma (i_1)} x_{\sigma (i_2)}
\cdots  x_{\sigma (i_l)}$, where $\sigma$ runs through $H$, is
considered. It is shown that $G$ has a free subgroup  of finite
index. For a field $K$, properties of the algebra $K[G]$ are
derived. In particular, the Jacobson radical $\mathcal{J}(K[G])$ is
always nilpotent, and in many cases the algebra $K[G]$ is
semiprimitive. Results on the growth and the Gelfand-Kirillov
dimension of $K[G]$ are given. Further properties of the semigroup
$S$ and the semigroup algebra $K[S]$ with the same presentation are
obtained, in case $S$ is cancellative. The Jacobson radical is
nilpotent in this case as well, and sufficient conditions for the
algebra to be semiprimitive are given.
\end{abstract}

\noindent {\it Keywords:} group ring, semigroup ring,
finitely presented, group, semigroup, Jacobson radical,
semiprimitive,  primitive. \\ {\it Mathematics
Subject Classification:} Primary 16S15, 16S34, 16S36,
20M05; Secondary  20M25, 16N20.
\section{Introduction}

Motivated by a general interest in finitely presented algebras with
homogeneous defining relations, in \cite{SymLengthL}, the authors
initiated the study of  the monoid algebras $K[S_{n,l}(H)]$,
where $l\geq 2$ is an integer and $H$ is a subgroup of $\Sym_n$ and
$S_{n,l}(H)$ is the monoid presented with generators $x_1,x_2,\dots
, x_n$ and with relations $x_{i_1}x_{i_2}\cdots x_{i_l} =x_{\sigma
(i_1)} x_{\sigma (i_2)} \cdots
 x_{\sigma (i_l)}$, for all $i_1,i_2,\dots ,i_l\in \{ 1,2,\dots ,n\}$ and all $\sigma\in
H$. In the present article, we obtain further results on algebras of
this type. Let $G_{n,l}(H)$ be the universal group of $S_{n,l}(H)$,
that is the group presented with the ``same'' generators and
relations. Recall that a subgroup $H$ of $\Sym_n$ is semi-regular if
the stabilizer of $i$ in $H$ is trivial for all $i\in\{ 1,2,\dots
,n\}$. In \cite[Theorems~2.3 and~2.5]{SymLengthL} it is shown that
$S_{n,l}(H)$ is cancellative if and only if $H$ is semi-regular and
abelian, and in this case $S_{n,l}(H)$  embeds in $G_{n,l}(H)$. For simplicity, we
will denote $S_{n,l}(H)$ by $S$, and $G_{n,l}(H)$ by $G$.

In Section~\ref{groupss}, we prove  that, for arbitrary $H$, the group $G$ has a free subgroup
of finite index. A similar result for $S$ was obtained in
\cite{SymLengthL}. In Section~\ref{algebrass}, we show that the
Jacobson radical, $\mathcal{J}(K[G])$, is nilpotent. We also prove
that if $H$ is transitive, then $K[G]$ is a Noetherian PI-algebra, and
the Gelfand-Kirillov dimension and the classical Krull dimension are
both equal to $1$. In Section~\ref{nilpotentt}, we show that if
$H$ is semi-regular and abelian, then $\mathcal{J}(K[S])$ is
nilpotent of index at most $|G/M|^{2}$, where $M$ is a normal
subgroup of finite index of $G$. This is a generalization of
\cite[Theorem~3.1]{SymLengthL}, in that the assumption that $H$ is
transitive is removed.

\section{Groups $G_{n,l}(H)$ defined by permutation relations}\label{groupss}
For a subgroup $H$ of $\Sym_n$  and an integer $l\geq 2$, let
$S_{n,l}(H)=\langle x_1,x_2, \ldots , x_n \mid
 x_{i_1}x_{i_2}\cdots x_{i_l} =x_{\sigma (i_1)} x_{\sigma (i_2)} \cdots  x_{\sigma (i_l)}, \;
 \sigma \in H, \;   i_1,\ldots ,i_l \in  \{ 1, \ldots , n \}
\rangle $. We consider the universal group $G_{n,l}(H)$ with the
same presentation as $S_{n,l}(H)$. In \cite{SymLengthL}, it was
shown that if $H$ is not abelian, then $S_{n,l}(H)$ is not embedded
 in $G_{n,l}(H)$.
\begin{lemma}\label{rewrite1}
In $G_{n,l}(H)$ the following equalities hold for an integer $ u
\geq 1$:
\begin{equation*}
\hspace{-10 pt}
\begin{array}{r@{}l}
x_{j_1}x_{j_2} \cdots x_{j_u}x_{k_u}^{-1}x_{k_{u-1}}^{-1} \cdots x_{k_1}^{-1}&{}=x_{\sigma(j_1)}x_{\sigma(j_2)}
\cdots x_{\sigma(j_u)}x_{\sigma(k_u)}^{-1}x_{\sigma(k_{u-1})}^{-1} \cdots x_{\sigma(k_1)}^{-1}, \vspace{3 pt} \\
 x^{-1}_{j_1}x^{-1}_{j_2} \cdots x^{-1}_{j_u}x_{k_u}x_{k_{u-1}} \cdots x_{k_1}&{}=x_{\sigma(j_1)}^{-1}x_{\sigma(j_2)}^{-1}
\cdots x_{\sigma(j_u)}^{-1}x_{\sigma(k_u)}x_{\sigma(k_{u-1})} \cdots x_{\sigma(k_1)},
\end{array}
\end{equation*}
where $j_1,j_2, \ldots , j_u, k_u, k_{u-1}, \ldots , k_1 \in \{ 1, 2, \ldots , n \}$.
\end{lemma}
\begin{proof}
Let $z$ be the smallest integer such that $zl > u$. Because of the
defining relations of $S_{n,l}(H)$, in its universal group for all
$\sigma \in H$ we have $x_{j_1}x_{j_2}\cdots x_{j_u} x_1^{zl-u}
= x_{\sigma(j_1)}x_{\sigma(j_2)}  \cdots
x_{\sigma(j_u)}x_{\sigma(1)}^{zl-u}$. Hence
$x_{\sigma(j_u)}^{-1}x_{\sigma(j_{u-1})}^{-1}  \cdots
x_{\sigma(j_1)}^{-1}x_{j_1}x_{j_2}\cdots x_{j_u} =
 x_{\sigma(1)}^{zl-u} x_1^{u-zl}$.
Therefore
\begin{equation*}
\begin{array}{r@{}l}
x_{\sigma(k_u)}^{-1}x_{\sigma(k_{u-1})}^{-1} \cdots
x_{\sigma(k_1)}^{-1} x_{k_1}x_{k_2}\cdots x_{k_u}
&{} = x_{\sigma(1)}^{zl-u}x_1^{u-zl} \vspace{3 pt} \\
&{} = x_{\sigma(j_u)}^{-1}x_{\sigma(j_{u-1})}^{-1}  \cdots x_{\sigma(j_1)}^{-1}x_{j_1}x_{j_2}\cdots x_{j_u},
\end{array}
\end{equation*}
 and hence
$$
 x_{j_1}x_{j_2}\cdots x_{j_u}  x_{k_u}^{-1}x_{k_{u-1}}^{-1}  \cdots x_{k_1}^{-1}  =
 x_{\sigma(j_1)}x_{\sigma(j_2)}  \cdots x_{\sigma(j_u)}x_{\sigma(k_u)}^{-1}x_{\sigma(k_{u-1})}^{-1}  \cdots
 x_{\sigma(k_1)}^{-1}.
$$
The other equality can be proved similarly.
\end{proof}
\begin{lemma}\label{rewrite2}
In $G_{n,l}(H)$ the following equalities hold:
\begin{equation*}
\hspace{-10 pt}
\begin{array}{r@{}l}
x_{\sigma_1(1)}x_{\sigma_2(1)} \cdots x_{\sigma_{l-1}(1)}x_{\sigma_l(j)}
&{}=x_1^{l-1}x_jx_1^{1-l}x_{\sigma_1^{-1}\sigma_2(1)}x_{\sigma_1^{-1}\sigma_3(1)}\cdots x_{\sigma_1^{-1}\sigma_l(1)},  \vspace{3 pt} \\
 x_{\sigma_1(1)}x_{\sigma_2(1)} \cdots x_{\sigma_{l-1}(1)}x^{-1}_{\sigma_l(j)}&{}=x_1^{l-1}x^{-1}_jx_1^{1-l}x_{\sigma_{l}\sigma_{l-1}^{-1}(1)}x_{\sigma_{l}\sigma_{l-1}^{-1}\sigma_{1}(1)}
 \cdots x_{\sigma_{l}\sigma_{l-1}^{-1}\sigma_{l-2}(1)},
\end{array}
\end{equation*}
where $j \in \{ 1, 2, \ldots , n \}$ and $\sigma_1,\dots
,\sigma_l\in H$.
\end{lemma}

\begin{proof} By the defining relations of $G_{n,l}(H)$ and
Lemma~\ref{rewrite1}, we have
\begin{equation*}
\begin{array}{r@{}l}
&{} \hspace{-18 pt} x_{\sigma_1(1)}x_{\sigma_2(1)} \cdots x_{\sigma_{l-1}(1)}x_{\sigma_l(j)} \vspace{3 pt}  \\
 &{}= x_{1}x_{\sigma_1^{-1}\sigma_2(1)} \cdots x_{\sigma_1^{-1}\sigma_{l-1}(1)}x_{\sigma_1^{-1}\sigma_l(j)}x_1^{1-l}x_1^{l-1} \vspace{3 pt}  \\
 &{}= x_{1}^2x_{\sigma_2^{-1}\sigma_3(1)} \cdots x_{\sigma_2^{-1}\sigma_{l-1}(1)}x_{\sigma_2^{-1}\sigma_l(j)}x_{\sigma_2^{-1}\sigma_1(1)}^{1-l}x_1^{l-1} \vspace{3 pt}  \\
 &{}= x_{1}^3x_{\sigma_3^{-1}\sigma_4(1)} \cdots x_{\sigma_3^{-1}\sigma_{l-1}(1)}x_{\sigma_3^{-1}\sigma_l(j)}x_{\sigma_3^{-1}\sigma_1(1)}^{2-l}x_{\sigma_2^{-1}\sigma_1(1)}^{-1}x_1^{l-1}\\
&{} \hspace{6 pt} \vdots \\
&{}= x_{1}^{l-1}x_{\sigma_{l-1}^{-1}\sigma_l(j)}x_{\sigma_{l-1}^{-1}\sigma_1(1)}^{-2}x_{\sigma_{l-2}^{-1}\sigma_1(1)}^{-1}\cdots x_{\sigma_2^{-1}\sigma_1(1)}^{-1}x_1^{l-1} \vspace{3 pt}  \\
&{}= x_{1}^{l-1}x_{j}x_{\sigma_{l}^{-1}\sigma_1(1)}^{-1}x_{\sigma_{l-1}^{-1}\sigma_1(1)}^{-1}\cdots x_{\sigma_2^{-1}\sigma_1(1)}^{-1}x_1^{l-1} \vspace{3 pt}  \\
&{}= x_{1}^{l-1}x_{j}x_1^{-1}x_{\sigma^{-1}_{1}\sigma_l\sigma_{l-1}^{-1}\sigma_1(1)}^{-1}\cdots x_{\sigma_{1}^{-1}\sigma_l\sigma_2^{-1}\sigma_1(1)}^{-1}x_{\sigma_{1}^{-1}\sigma_l(1)}^{l-1} \vspace{3 pt}  \\
&{}=x_{1}^{l-1}x_{j}x_1^{-2}x_{\sigma^{-1}_{1}\sigma_{l-1}\sigma_{l-2}^{-1}\sigma_1(1)}^{-1}\cdots
x_{\sigma_{1}^{-1}\sigma_{l-1}\sigma_2^{-1}\sigma_1(1)}^{-1}x_{\sigma_{1}^{-1}\sigma_{l-1}(1)}^{l-2}
x_{\sigma_{1}^{-1}\sigma_l(1)}\\
&{} \hspace{6 pt} \vdots \\
 &{}=x_1^{l-1}x_jx_1^{1-l}x_{\sigma_1^{-1}\sigma_2(1)}x_{\sigma_1^{-1}\sigma_3(1)}\cdots
 x_{\sigma_1^{-1}\sigma_l(1)},
\end{array}
\end{equation*}
and the first equality is proved. Now  applying the first equality
for
$\sigma_{l}\sigma_{l-1}^{-1},\sigma_{l}\sigma_{l-1}^{-1}\sigma_{1},\ldots,
\sigma_{l}\sigma_{l-1}^{-1}\sigma_{l-2}, \sigma_l\in H$, we obtain
$$ x_{\sigma_{l}\sigma_{l-1}^{-1}(1)}x_{\sigma_{l}\sigma_{l-1}^{-1}\sigma_{1}(1)} \cdots x_{\sigma_{l}\sigma_{l-1}^{-1}\sigma_{l-2}(1)}x_{\sigma_{l}(j)}
=x_1^{l-1}x_jx_1^{1-l}x_{\sigma_{1}(1)}x_{\sigma_{2}(1)}
\cdots x_{\sigma_{l-1}(1)},$$ and thus the second equality follows.
\end{proof}
\begin{theorem}\label{formmmm}
 Assume that the different orbits under the action of $H$ are $X_{k
_1},X_{k_2}, \ldots , X_{k_r}$, where $X_k$ denotes the orbit of
$k$, and $k_1=1$. Then the  subgroup $F$ of $G_{n,l}(H)$
generated by $x_1, x_{k_2}, \ldots , x_{k_r}$ is free of rank $r$,
and for every $g\in G_{n,l}(H)$, there  exist $w \in F$ and
$\sigma_1, \sigma_2 , \ldots , \sigma_{l-1} \in H$ such that
$$ g = wx_{\sigma_1(1)} x_{\sigma_2(1)} \cdots x_{\sigma_{l-1}(1)} .$$
\end{theorem}
\begin{proof} Let $G=G_{n,l}(H)$.
Let $G_1$ be the normal closure of $\gr(x_{k_j}x_{ \sigma(k_j)}^{-1} \mid k_j \in \{1, k_2, \ldots ,k_r \}, \sigma \in H)$ in $G$. We consider
$\bar{G}= G / G_1$. Then $\bar{G} = \gr(x_1, x_2 , \ldots , x_n \mid x_{i_1}x_{i_2} \cdots x_{i_l} = x_{\sigma(i_1)} x_{\sigma(i_2)}
\cdots x_{\sigma(i_l)}, \ x_{k_j}=x_{\sigma(k_j)}, \;
 \sigma \in H, \;   i_1,\ldots ,i_l \in  \{ 1, \ldots , n \} , \; k_j \in  \{ 1, k_2, \ldots , k_r \} )$. The second
relation in the presentation of $\bar G$ implies the first, and
$\bar{G} = \gr(X_{k_1}, X_{k_2}, \ldots , X_{k_r})$ is a free group
of rank $r$. Consider the natural epimorphism $\pi : G \to \bar{G}$.
If there was a non-trivial relation in $G$, of type $w=1$ with $w$ a
reduced word  in $\{ x_1^{\pm 1}, x_{k_2}^{\pm 1}, \ldots ,
x_{k_r}^{\pm 1} \}$, then we would have $\pi(w)=1$, a contradiction
because $\bar{G}$ is a free group. Thus the first part of the result
follows.

For the second part, we will use induction on the length  of the
reduced words in $\{ x_1^{\pm 1}, x_{2}^{\pm 1}, \ldots ,
x_{n}^{\pm 1} \}$ that represent an element $g\in G$. For length
$0$ we have $1=x_1^{1-l}x_1^{l-1}$ and for length $1$, according to
Lemma \ref{rewrite1} and the defining relations, we have
\begin{equation*}
\begin{array}{r@{}l}
x_{\sigma(k)}&{} =  x_{\sigma(k)}x_1^{-1}x_1^{2-l}x_1^{l-2}x_1=
x_{k}x_{\sigma^{-1}(1)}^{-1}x_1^{2-l}x_1^{l-2}x_1 \vspace{3 pt}  \\
&{} =x_{k}x_{1}^{-1}x^{2-l}_{\sigma(1)}x_{\sigma(1)}^{l-2}x_{\sigma(1)}
=x_{k}x_{1}^{-1}x^{2-l}_1x_1^{l-2}x_{\sigma(1)},
\end{array}
\end{equation*}
$$x_{\sigma(k)}^{-1}=x_{\sigma(k)}^{-1}x_1^{1-l}x_1^{l-1}= x_k^{-1}
x_{\sigma^{-1}(1)}^{1-l} x_1^{l-1}= x_k^{-1}
x_1^{1-l}x_{\sigma(1)}^{l-1}.$$ Let $s\geq 1$ and assume the result
is true for all $g\in G$ such that the length of some reduced word
in $\{ x_1^{\pm 1}, x_{2}^{\pm 1}, \ldots , x_{n}^{\pm 1} \}$
that represents $g$ is $s$. Let
$$ x_{\sigma_1(j_1)}^{\epsilon_1} x_{\sigma_2(j_2)}^{\epsilon_2} \cdots x_{\sigma _{s+1}(j_{s+1})}^{\epsilon_{s+1}},$$
be a reduced word with $\epsilon_u \in \{1, -1\}$,  $j_u \in \{1,
k_2 , \ldots , k_r \} $ and $\sigma _u\in H$, for $1 \leq u \leq
s+1$. By induction hypothesis there exist $w\in F$ and $\tau_1,\dots
,\tau_{l-1}\in H$ such that
\begin{eqnarray*}
x_{\sigma_1(j_1)}^{\epsilon_1} x_{\sigma_2(j_2)}^{\epsilon_2} \cdots
x_{\sigma _{s+1}(j_{s+1})}^{\epsilon_{s+1}}=wx_{\tau_1(1)}\cdots
x_{\tau_{l-1}(1)}x_{\sigma _{s+1}(j_{s+1})}^{\epsilon_{s+1}}.
\end{eqnarray*}
Now by Lemma~\ref{rewrite2}, the result follows.
\end{proof}

\section{The algebra $K[G_{n,l}(H)]$}\label{algebrass}

In this section we study the algebraic structure of the group
algebra $K[G_{n,l}(H)]$ over a field $K$, for any subgroup $H$ of
$\Sym_n$. We denote $G_{n,l}(H)$ by $G$.

\begin{theorem}
Let $F$ be the free subgroup of $G$ defined in the statement of
Theorem \ref{formmmm}. Then $[G:F]\leq |H|^{l-1}$ and
$(\mathcal{J}(K[G]))^{[G:F]}=0$. Furthermore, if the characteristic of $K$
does not divide $[G:F]$, then $\mathcal{J}(K[G])=0$.
\end{theorem}
\begin{proof}
By Theorem \ref{formmmm}, we know that  $[G:F]\leq |H|^{l-1}$. Then by
a result of Villamayor
%(Theorem~\ref{passman-T7.2.7}),
\cite[Theorem~7.2.7]{passman},
$(\mathcal{J}(K[G]))^{[G:F]} \subseteq (\mathcal{J}(K[F]))\cdot K[G]$. Because $F$ is a
free group, $\mathcal{J}(K[F])=0$, and hence $(\mathcal{J}(K[G]))^{[G:F]}=0$. If the
characteristic of $K$ does not divide $[G:F]$, then by the same result
%(Theorem~\ref{passman-T7.2.7}),
\cite[Theorem~7.2.7]{passman},
we have $\mathcal{J}(K[G])=
(\mathcal{J}(K[F]))\cdot K[G]= 0$.
\end{proof}

\begin{theorem}
\begin{itemize}
\item[(a)] If $H$ is transitive, then $K[G]$ is a Noetherian PI-algebra and $$ \GKdim (K[G])=1= \clKdim (K[G]).$$
\item[(b)] If $H$ is not transitive, then $K[G]$ has exponential growth.
\end{itemize}
\end{theorem}
\begin{proof}
\begin{itemize}
\item[(a)]
If $H$ is transitive, by Theorem \ref{formmmm}, $K[G]$ is a finitely
generated module over $K[\gr(x_1)]=K[x_1,x_1^{-1}]$. Thus $K[G]$ is
a Noetherian PI-algebra and $\GKdim (K[G])=1$.  By
%Theorem~\ref{book-jan-T23.4},
\cite[Theorem~23.4]{book-jan},
$ \clKdim (K[G])=\GKdim (K[G])=1$.
\item[(b)] If $H$ is not transitive, by Theorem~\ref{formmmm},  $G$ contains a free subgroup of rank $2$. Thus $K[G]$ has exponential growth.
\end{itemize}
\end{proof}

\section{The algebra $K[S_{n,l}(H)]$}\label{nilpotentt}
In this section, we obtain a generalization of \cite[Theorem
3.1]{SymLengthL}.

To do so, we recall Lemma 2.2 and Theorem 2.5  of \cite{SymLengthL}. The first part is the monoid version of Theorem~\ref{formmmm}.

\begin{proposition}\label{semi}
Let $H$ be a subgroup of $\Sym_n$. Then $S=S_{n,l}(H)$
contains a free submonoid $\FM$, with basis $\{ x_1, x_{k_2}, \ldots ,x_{k_r}\}$
(one element from each orbit on the generating set), and a finite
subset $T$ such that $S=\bigcup_{t\in T} \FM t =\bigcup_{t\in
T}t\FM$. If, furthermore, $H$ is abelian and semi-regular, then $S$
is a submonoid of its universal group $G=G_{n,l}(H)$.
\end{proposition}

Recall that a semigroup $P$ is said to be right reversible if $Pa\cap
Pb\neq\emptyset$ for all $a,b\in P$.

We need the following elementary technical lemma.

\begin{lemma}\label{ReversibleSub}
A right reversible submonoid of a free monoid is contained in
a cyclic group.
\end{lemma}
% Let $\FM = \langle x_{k_1}, x_{k_2}, \ldots , x_{k_r} \rangle $ be a free monoid and $B$ a reversible submonoid of $\FM$. Then $B$ is contained in a cyclic  group.
\begin{proof}
Let $B$ be a right reversible submonoid of a free monoid
$\FM$. Let $|x|$ denote the length of  $x$ in a chosen basis $X$ of
the free monoid $\FM$. Let $s,t \in B$. Then $Bs\cap Bt\neq
\emptyset$. Thus $as=bt$ for some $a,b\in B$.
%Assume $|s|\geq |t|$, with $|x|$ the length of
Because $\FM$ is free, it follows that if
    $|s|\geq |t|$,
then
   $s=vt$
for some $v\in \FM$.
In particular, if $x,y\in B$ and
  $|x|=|y|$, then $x=y$.

Clearly, there exist positive integers $n$ and $m$ such that
   $|s^n| =|t^m|$.
Since $s^n , t^m\in B$, it follows that $s^n=t^m$.

Let $\FG$ denote the free group with basis $X$ .
Consider the group
  $C=   \langle s, s^{-1}, t, t^{-1}\rangle \subseteq \FG  $.
Because a subgroup of a free group is free, we get that $C$ is a free group of rank at most two.
As $s^n=t^m$, the rank has to be one and thus $C$ is cyclic.
So, any two elements of $B$ commute and thus $B$ generates an abelian subgroup  in the  free group in $\FG$. So, this group is cyclic.
\end{proof}

Because of Theorem~\ref{formmmm},  $G=G_{n,l}(H)$ has a free
subgroup of finite index. Let $M$ be a normal subgroup of $G$ that
is of finite index and a free group.

\begin{theorem}\label{Non-Trans}
Let $H$ be an abelian and semi-regular subgroup of $\Sym_n$.  Let
$G=G_{n,l}(H)$ be the universal group of $S=S_{n,l}(H)$. Then,
$\mathcal{J}(K[S])$ is nilpotent of index at most $|G/M|^{2}$. If
$\ch (K) = 0$ or $\ch (K) = p$ and  $p \nmid |G/M|$, then
$\mathcal{J}(K[S])=0$ and $\mathcal{J}(K[G])=0$.
\end{theorem}

\begin{proof}
Obviously, one can consider $K[G]$ in a natural way as a $G/M$
graded ring. Its component of degree $e$ (the identity of $G/M$) is
$K[M]$.  By Proposition~\ref{semi}, $K[S]$ is a homogeneous
subring of $K[G]$ and hence it also is a $G/M$-graded ring, with
component of degree $e$ the subring $K[M\cap S]$. As a subsemigroup
of $M$, the semigroup $M\cap S$ is a two-unique product monoid and
thus $\mathcal{J}(K[S\cap M])=0$. Since $K[S]$ is graded by a finite
group, it then follows hat
%from
%Now we consider our semigroup algebra $K[S]$. It is naturelly graded by the finite group $G/\FG$.
%The component of degree $e$ is $K[S\cap \FG]=K[\FM]$ and this has trivial Jacobson radical. Hence,
$\mathcal{J}(K[S])^{m^2}=0$, where $m=|G/M|$  (this follows from
\cite{CohMont1,CohMont2} and Lemma 20 in \cite{Clase}). This proves
the first part of the theorem.

To prove the second part, assume $\ch (K) = 0$ or $\ch (K) = p$ and
$p \nmid m$. We prove the result by contradiction. So, assume
$\mathcal{J}(K[S])\neq 0$. Then,  by
\cite[Corollary~1]{alg-cancel-sg}, there exits a left and right
reversible submonoid $P$ of $S$ such that $\mathcal{J}(K[P])\neq 0$.

We claim that  $P\cap M$ also is a right reversible monoid.
Indeed, let $s,t\in P\cap M$. Then $Ps\cap Pt \neq \emptyset$ and thus
   $p_1 s =p_2 t$,
for some $p_1, p_2 \in P$. Hence, $(p_{2}^{m-1} p_1)s =
(p_{2}^{m})t$. Clearly, $p_{2}^{m}\in P\cap M$ and thus
$p_{2}^{m-1}p_1 \in M\cap P$.  Thus, $(P\cap M)s \cap (P\cap
M)t\neq \emptyset$. So, indeed,  $P\cap M$ is  right reversible.

By Lemma~\ref{ReversibleSub}, $P\cap M$ is contained in   a cyclic subgroup of $M$.
Let $Q=PP^{-1}$ be the group of fractions of $P$.
Then $Q\cap M$ is normal in $Q$ and $Q\cap M$ is a cyclic group.
Let $p\in P$. Clearly,  $p$ acts by conjugation on $Q\cap M$.
Since the latter is cyclic, this conjugation either is the trivial map or it maps an element to its inverse.
Consequently, if $f\in P\cap M$ and $p\in P$, then
    $pf=f^{-1}p$     or $pf =fp$.
By a degree argument,  the former is impossible. So $pf=fp$. Hence,
$P\cap M$ is central in $P$ and  thus $Q\cap M=(P\cap M)(P\cap
M)^{-1}$ is a central subgroup of $Q$ of finite index. Actually the
index is a divisor of $m$.

Hence $Q$ is a finite conjugacy group, with an infinite cyclic
central subgroup of finite index that is a divisor of $m$. By
\cite[Lemma 4.1.4]{passman} such a group has a finite normal
subgroup, say $N$, such that $Q/N$ is cyclic. Clearly, $N$ is
isomorphic to a subgroup of $G/M$ and thus $|N| $ is a divisor of
$m$. Hence, by \cite[Theorem 7.3.1]{passman},
$\mathcal{J}(K[Q])\subseteq \mathcal{J}(K[N])K[Q]$. Moreover, by
assumption, either $\ch (K)=0$ or $p \nmid m$. So, by Maschke's
Theorem,  $\mathcal{J}(K[Q])=0$. Further, by
\cite[Corollary~11.5]{book-jan}, $\mathcal{J}(K[P])
=\mathcal{J}(K[Q]) \cap K[P]$.  So $\mathcal{J}(K[P]) =0$, a
contradiction.
\end{proof}

\vspace{30pt}
 \noindent \begin{tabular}{llllllll}
 F. Ced\'o && E. Jespers  \\
 Departament de Matem\`atiques &&  Department of Mathematics \\
 Universitat Aut\`onoma de Barcelona &&  Vrije Universiteit Brussel  \\
08193 Bellaterra (Barcelona), Spain    && Pleinlaan
2, 1050 Brussel, Belgium \\
 cedo@mat.uab.cat && efjesper@vub.ac.be\\
   &&   \\
G. Klein &&  \\ Department of Mathematics &&
\\  Vrije Universiteit Brussel && \\
Pleinlaan 2, 1050 Brussel, Belgium &&\\ gklein@vub.ac.be&&
\end{tabular}
\end{document}